\documentclass[a4paper,11pt,leqno]{amsart}
\usepackage{amsmath}
\usepackage{amsthm}
\usepackage{amsfonts}
\usepackage{amssymb}

\newtheorem{theorem}{Theorem}[section]{\bf}{\it}
\newtheorem{mainth}{Theorem}{\bf}{\it}
\newtheorem{lemma}[theorem]{Lemma}{\bf}{\it}
\newtheorem{proposition}[theorem]{Proposition}{\bf}{\it}
{\bf}{\it}

\theoremstyle{definition}

\newtheorem{rem}[theorem]{Remark}

\numberwithin{equation}{section}

\newcommand{\rd}{\mathrm{d}} 

\newcommand{\R}{\mathbb R}

\newcommand{\N}{\mathbb N}

\newcommand{\loc}{{\operatorname{loc}}}

\newcommand{\spt}{\operatorname{spt}}

\newcommand{\id}{{\operatorname{id}}}

\newdimen\vintkern\vintkern11pt
\def\vint{-\kern-\vintkern\int}

\newcommand{\norm}[1]{\lVert #1 \rVert}

\newcommand{\grad}{\nabla}

\newcommand{\bS}{\mathbb{S}}

\newcommand{\bM}{\mathbb{M}}

\newcommand{\haus}{\mathcal{H}}

\newcommand{\dmu}{\;\mathrm{d}\mu}

\newcommand{\cA}{\mathcal{A}}

\newcommand{\vol}{\mathrm{vol}}

\newcommand{\weakto}{\rightharpoonup}

\begin{document}

\title{Equilibrium measures for uniformly quasiregular dynamics}
\date{\today}

\author{Y\^usuke Okuyama}
\address{Division of Mathematics, Kyoto Institute of Technology, Sakyo-ku, Kyoto 606-8585, JAPAN
}
\email{okuyama@kit.ac.jp}
\thanks{Y. O. is partially supported by JSPS Grant-in-Aid for Young Scientists (B), 21740096.}

\author{Pekka Pankka}
\address{University of Helsinki, Department of Mathematics and Statistics (P.O. Box 68), FI-00014 University of Helsinki, Finland}
\email{pekka.pankka@helsinki.fi}
\thanks{P. P. is partially supported by the Academy of Finland projects \#126836 and \#256228.}



\begin{abstract}
We establish the existence and fundamental properties of the equilibrium measure
in uniformly quasiregular dynamics. 
We show that a uniformly quasiregular endomorphism $f$ of degree at least $2$ 
on a closed Riemannian manifold admits an equilibrium measure $\mu_f$,
which is balanced and invariant under $f$ and non-atomic, and whose support agrees with
the Julia set of $f$. Furthermore we show that $f$ is strongly
mixing with respect to the measure $\mu_f$. 
We also characterize the measure $\mu_f$ using an approximation property by 
iterated pullbacks of points under $f$ up to a set of exceptional initial points
of Hausdorff dimension at most $n-1$. 
These dynamical mixing and approximation results are reminiscent of
the Mattila-Rickman equidistribution theorem 
for quasiregular mappings. Our methods are based on 
the existence of an invariant measurable conformal structure 
due to Iwaniec and Martin and the $\cA$-harmonic potential theory. 
\end{abstract} 
 
\maketitle

\section{Introduction}
\label{sec:introduction}

By the classical Liouville theorem, conformal mappings on the standard $n$-sphere $\bS^n$ are M\"obious transformations for $n\ge 3$. Quite surprisingly, Iwaniec and Martin showed in \cite{IwaniecT:Quas} that there exist measurable conformal structures on the $n$-sphere $\bS^n$ admitting conformal mappings that are not local homeomorphisms. These measurable conformal structures are quasiconformally equivalent to the standard conformal structure on $\bS^n$ and these conformal mappings, that plays the r\^ole in rational maps, are \emph{uniformly quasiregular} with respect to the standard structure.

Recall that a continuous mapping $f\colon \bM \to \bM$ on an oriented Riemannian $n$-manifold $\bM$ is \emph{$K$-quasiregular}, for some $K\ge 1$, if $f$ belongs to the local Sobolev space $W^{1,n}_\loc(\bM,\bM)$ and satisfies the distortion inequality
\begin{equation}
\label{eq:K}
\norm{Df}^n \le K J_f\quad \mathrm{a.e.}\ \bM,
\end{equation}
where $\norm{Df}$ is the point-wise operator norm of the differential $Df$ of the mapping $f$ and $J_f$ is the Jacobian determinant of the differential.
A quasiregular endomorphism $f\colon \bM \to \bM$ is said to be \emph{uniformly quasiregular} if there exists $K\ge 1$ so that all the iterates of $f$ are $K$-quasiregular.

Results of Peltonen \cite{PeltonenK:Exauqm} and Astola, Kangaslampi, and Peltonen \cite{AstolaL:Latmcm} show that the existence of branching uniformly quasiregular dynamics is not limited to $\bS^n$. However, results of Martin, Meyer, and Peltonen \cite{MartinG:GenLp} show that, among space forms, only spherical space forms admit branching uniformly quasiregular endomorphisms; see \cite{MartinG:GenLp} for the terminology and for related results. Since the existence question of branching uniformly quasiregular dynamics is interesting also for manifolds that are not space forms, we state our results for general closed manifolds.

To state our main results, we recall that a Borel measure $\mu$ on $\bM$ is \emph{balanced under $f$} if 
\[
f^* \mu = (\deg f)\mu,
\]
where $f^*\mu$ is the pull-back of $\mu$ under $f$; see Section \ref{sec:construct} for details. We note that if $\mu$ is balanced under $f$, then it is \emph{invariant} under $f$, that is, $f_*\mu=\mu$.
 
Our main result is the following existence theorem for an invariant measure; the approximation property of $\mu_f$ is reminiscent to the complex analytic case, see Lyubich \cite{LjubichM:Entpre}. 

In all of the following statements, we assume that $\bM$ is closed, connected, and oriented Riemannian $n$-manifold.
Here $\weakto$ means the weak convergence of measures on $\bM$.

\begin{mainth}
\label{thm:existence_approximability}
Let $f\colon \bM\to \bM$ be a uniformly quasiregular endomorphism of degree $d>1$. 
Then there exists an invariant measure $\mu_f$ satisfying 
the approximation $($or equidistribution$)$ property 
\[
\frac{(f^k)^*\delta_a}{d^k}\weakto \mu_f
\]
as $k\to\infty$
for every $a\in \bM$ except for a subset of Hausdorff dimension at most $n-1$ in $\bM$. 
\end{mainth}

We prove Theorem \ref{thm:existence_approximability} in two parts: first the existence in Theorem \ref{thm:limits_Lp} and then the approximation (or equidistribution) property in Theorem \ref{th:equidist}. 

We call the measure $\mu_f$ is an \emph{equilibrium measure of $f$} 
since this measure has properties familiar from complex dynamics.
Here, the Julia set of $f$ is defined by
\begin{gather*}
 \mathcal{J}(f):=\{x\in\bM;\{f^k;k\in\mathbb{N}\}\text{ is not normal on any open ball around }x\}.
\end{gather*}

\begin{mainth}
\label{thm:properties}
Let $f\colon \bM\to \bM$ be a uniformly quasiregular endomorphism of degree $d>1$.
Then the equilibrium measure $\mu_f$ is non-atomic and balanced under $f$, 
and the support $\spt\mu_f$ agrees with $\mathcal{J}(f)$.
\end{mainth}

As our final result, we show that $f$ is mixing with respect to $\mu_f$. 
This result and the approximation property with respect to $\mu_f$ 
in Theorem \ref{thm:existence_approximability} can be viewed as dynamical counterparts of the Mattila-Rickman equidistribution theorem \cite[Theorem 5.1]{MattilaP:Avecfq}.

\begin{mainth}
\label{thm:mixing}
Let $f\colon \bM\to \bM$ be a uniformly quasiregular endomorphism of degree $d>1$.
Then $f$ is strongly mixing with respect to the invariant measure $\mu_f$, that is, for every $\phi,\psi\in L^2(\mu_f)$,
\begin{gather}
\label{eq:mixing}
 \lim_{k\to\infty}\int_{\bM}(\phi\circ f^k)\psi\rd\mu_f
=\int_{\bM}\psi\rd\mu_f\int_{\bM}\phi\rd\mu_f.
\end{gather}
In particular, $\mu_f$ is ergodic under $f$. 
\end{mainth}

The proofs of Theorems \ref{thm:existence_approximability} and \ref{thm:mixing} are, in spirit, potential theoretic. Given a manifold $\bM$ and a uniformly quasiregular endomorphism $f$ of $\bM$, we interpret the invariant conformal structure of Iwaniec and Martin in terms of the given Riemannian structure. This translates the $n$-Laplace operator $\Delta_n$ of the measurable conformal structure to an $f$-invariant $\cA$-harmonic equation. Although the $n$-Laplace operator $\Delta_n$ and the corresponding $\cA$-harmonic operator are non-linear, the methods in \cite{DrasinD:Equnt} are within our reach. 


\section{Preliminaries}

Let $\bM$ be a closed, oriented, and connected Riemannian $n$-manifold, $n\ge 2$. 
We denote the Riemannian metric on $\bM$ by $\langle \cdot, \cdot \rangle$ and the induced distance between points $x$ and $y$ on $\bM$ by $|x-y|$. Given $x\in \bM$ and $r>0$, we denote by $B(x,r)$ the open ball of radius $r$ about $x$ in this metric. The corresponding closed ball is $\bar B(x,r)$. 
We denote by $\vol_\bM$ the volume form induced by this metric. We denote by $\haus^s$ the \emph{Hausdorff $s$-measure on $\bM$} associated to the distance.

We denote by $L^p(\bigwedge^m \bM)$ the space of $p$-integrable $m$-forms on $\bM$ and by $W^{1,p}(\bigwedge^m \bM)$ the Sobolev space of $p$-integrable $m$-forms having (weak) exterior derivative in $L^p(\bM)$; see \cite{IwaniecT:Nonhtm} for details. We denote also by $\star$ the \emph{Hodge star} associated to the Riemannian metric and by $d^*$ the \emph{co-exterior derivative} on $\bM$, that is, the operator $d^* = (-1)^{m(n-m)} \star d\star$ on $m$-forms.

Given a measure $\mu$ on $\bM$, we denote the \emph{support of $\mu$}, i.e.\;the smallest closed set $E$ satisfying $\mu(\bM\setminus E)=0$, by $\spt \mu$.

Throughout the article we consider quasiregular endomorphisms $f$ on closed Riemannian $n$-manifolds $\bM$, that is, continuous mappings $\bM\to \bM$ in the Sobolev class $W^{1,n}(\bM,\bM)$ satisfying the distortion inequality \eqref{eq:K}. 
We define the Sobolev space $W^{1,n}(\bM,\bM)$ as in \cite{HajlaszP:Weadmb}, that is, we fix a smooth Nash embedding $\iota\colon \bM \to \R^N$, and say that $f \in W^{1,n}(\bM,\bM)$ if coordinates of $\iota \circ f$ are Sobolev functions in $W^{1,n}(\bM)$. 

The \emph{push-forward} $f_*\psi \colon \bM \to \R$ of $\psi\in C^0(\bM)$ by a quasiregular endomorphism $f$ on $\bM$ is defined by
\[
f_*\psi(y):= \sum_{x \in f^{-1}(y)} i(x,f) \psi(x),
\]
where $i(x,f)$ is the \emph{local index of $f$} at $x$; see e.g.\;\cite[Section 4]{RickmanS:Quam}. In particular, $f_*\psi$ is continuous and belongs to $W^{1,n}(\bM)$; see \cite[Lemmas 14.30 and 14.31]{HeinonenJ:Nonptd}.
Given a Borel measure $\mu$ on $\bM$ the pull-back measure $f^*\mu$ is defined by the formula
\[
\int_{\bM}\phi \rd(f^*\mu)=\int_{\bM}f_*\phi\rd\mu,
\]
where $\phi \in C^0(\bM)$.

In the proof of Theorem \ref{thm:existence_approximability} we use  both classical potential theory of measures and non-linear potential theory of $\cA$-harmonic equations. We recall now some classical facts on potential theory that will be used in the following sections. 

Given a finite Borel measure $\mu$ on $\bM$, we denote by $U^\mu(x)$ the 
(Riesz) $1$-potential of $\mu$ at $x\in \bM$, that is,  
\[
U^\mu(x) = \int_\bM \frac{\rd\mu(y)}{|x-y|^{n-1}}
\]
for $x\in \bM$.

Given a compact set $E\subset \bM$, the energy functional
\begin{equation}
\label{eq:I_1}
I_1(\mu) = \int_E \int_E \frac{\rd\mu(x)\rd\mu(y)}{|x-y|^{n-1}}
\end{equation}
admits the unique minimizer $\lambda_E$ among all probability 
Borel measures $\mu$ on $E$
if $\sup_{\mu}I_1(\mu)<\infty$, where $\mu$ ranges over all probability
Borel measures on $E$; see e.g.\;\cite[Chapter II \S 1]{LandkofN:Foumpt}. We call the measure $\lambda_E$ the \emph{equilibrium measure of $E$ with respect to the Riesz $1$-capacity} and denote by $W_1(E)$ the minimal energy $I_1(\lambda_E)$. The \emph{Riesz $1$-capacity} $C_1(E)$ of the compact set $E$ is
\[
C_1(E) = \frac{1}{W_1(E)}.
\]
We have 
\begin{gather}
 U^{\lambda_E}(x) \le \frac{1}{C_1(E)}\label{eq:upper}
\end{gather}
for all $x\in \bM$; in fact, $U^{\lambda_E}\equiv C_1(E)^{-1}$ in $\spt\mu\setminus E'$, where $C_1(E')=0$; see e.g.\;\cite[p.\ 137]{LandkofN:Foumpt} for more details.


\section{ $\cA$-harmonic potentials}

In this section we prove the following proposition on the existence of solutions to certain $\cA$-harmonic Poisson-type equations.

\begin{proposition}\label{prop:standard_qr}
Let $f\colon \bM \to \bM$ be a $K$-quasiregular endomorphism of degree $d>0$. 
Then there exists $p = p(n,K)>1$ so that
for every $\omega\in L^p(\bigwedge^n\bM)$, the following $\cA$-harmonic equation
\begin{gather}
\label{eq:A_harm_omega}
 \rd^* \cA(\rd u) = \star \left(\frac{f^*\omega}{d}-\omega\right)
\end{gather}
has a solution $u\in W^{1,n}(\bM)$.
\end{proposition}

In \eqref{eq:A_harm_omega}, $\cA$ is a uniformly elliptic measurable bundle map $\cA \colon T^*\bM \to T^*\bM$ so that there exist $a>0$ and $b>0$ satisfying the following conditions:
\begin{enumerate}
\item $\langle \cA(\xi) - \cA(\zeta), \xi-\zeta\rangle \ge a (|\xi|+|\zeta|)^{n-2}|\xi-\zeta|^2$,
\item $|\cA(\xi)-\cA(\zeta)| \le b (|\xi|+|\zeta|)^{n-2} |\xi-\zeta|$, and
\item $\cA(t\xi) = t|t|^{n-2}\cA(\xi)$
\end{enumerate}
for all $\xi,\zeta \in T^*_x \bM$, $t \in \R$, and for almost all $x\in \bM$. Furthermore, 
we assume that for every measurable $1$-form $\omega$ on $\bM$,
so is $x \mapsto \cA_x(\omega)$.
We refer \cite{HeinonenJ:Nonptd} for a detailed discussion on the non-linear potential theory of $\cA$-harmonic equations.

The proof of Proposition \ref{prop:standard_qr} is based on the higher integrability of quasiregular mappings and the following lemma on solvability of $\cA$-harmonic Poisson equations.

\begin{lemma}\label{lem:standard}
Let $p>1$ and let $\omega \in L^p(\bigwedge^n T\bM)$ be an $n$-form so that 
\begin{gather}
\int_{\bM} \omega = 0.\label{eq:weakly_exact}
\end{gather}
Then there exists a solution $u \in W^{1,n}(\bM)$ to the $\cA$-harmonic equation
\begin{gather}
 \rd^* \cA(\rd u) = \star \omega,\label{eq:cohomological}
\end{gather}
\end{lemma}

\begin{proof}
By \eqref{eq:weakly_exact}, $\omega$ is weakly exact and there exists an $(n-1)$-form 
$\tau \in L^p(\bigwedge^{n-1}\bM)$ so that $\rd\tau = \omega$. Moreover, by the Poincar\'e inequality for differential forms, there exits a weakly closed $(n-1)$-form $\tau_0$ so that $\tau-\tau_0 \in W^{1,p}(\bigwedge^{n-1}\bM)$; see \cite[Theorem 6.4]{IwaniecT:Nonhtm}. Since $\rd(\tau-\tau_0) = \rd\tau = \omega$, we may assume that $\tau \in W^{1,p}(\bigwedge^{n-1} \bM)$.
Since 
\[
\rd^*\cA(\rd u) = \star \omega = \star \rd \tau,
\]
it suffices to show that, up to a sign, there exists $u$ so that 
\[
\star \cA(\rd u) = \tau.
\]
Since $p>1$, the Sobolev embedding theorem concludes $\tau \in L^{p^*}(\bigwedge^{n-1}\bM)$, where $p^* := np/(n-p)$. Since $p^* > n/(n-1)$ and $\cA \colon L^n(\bigwedge^1 \bM) \to L^{n/(n-1)}(\bigwedge^1 \bM)$ is surjective by the Minty-Browder theorem, there exists a solution $u \in W^{1,n}(\bM)$; see \cite[Section 8]{IwaniecT:Nonhtm} and, e.g., \cite[Sections 2.3 and 2.4]{DOnofrioL:Notpha} for details.
\end{proof}

\begin{proof}[Proof of Proposition $\ref{prop:standard_qr}$]
By change of variables,
\[
\int_{\bM}\frac{f^* \omega}{d} = \int_{\bM} \omega
\]
for each $\omega\in L^1(\bigwedge^n\bM)$; see e.g.\;\cite{MartioO:Lusin}. Thus, by Lemma \ref{lem:standard}, it suffices to show that for each $s>1$,
there exists $p=p(n,K,s)>1$ so that $f^*\omega \in L^s(\bigwedge^n \bM)$ if $\omega \in L^p(\bigwedge^n \bM)$; this is a direct consequence of the higher integrability of derivatives of quasiregular mappings.

Since $f$ is $K$-quasiregular, there exists $r=r(n,K)$ so that $f\in W^{1,r}(\bM,\bM)$; see e.g.\;\cite{MartioO:Onidq,MeyersN:Resrfs,IwaniecT:LpiPDE}. In particular, $|Df| \in L^r(\bM)$. Since $J_f \le |Df|^n$ almost everywhere, we have $J_f \in L^q(\bM)$, where $q=q(n,K):=r/n > 1$.

Thus for every $s>1$, there exists $p=p(q(n,K),s)>1$ so that $f^*\omega \in L^s(\bigwedge^n \bM)$ for all $\omega \in L^p(\bigwedge^n \bM)$. Indeed, fix $s>1$, and then there exists $\beta=\beta(n,K,s)\in(0,1)$ so that $\beta q = (s-\beta)(1-\beta)$. Let $\omega \in L^p(\bigwedge^n \bM)$, where $p:=s/\beta$. Then $\omega = u \vol_{\bM}$, where $u\in L^p(\bM)$. 
Since $f^*\omega=(u\circ f)J_f\vol_{\bM}$, we have
\begin{eqnarray*}
\int_{\bM} (u\circ f)^s J_f^s \vol_{\bM} &=& \int_{\bM} (u\circ f)^s J_f^\beta J_f^{s-\beta} \vol_{\bM} \\
&\le& \left( \int_{\bM} (u\circ f)^{s/\beta} J_f \vol_{\bM} \right)^{\beta} \left( \int_{\bM} J_f^{(s-\beta)\frac{1-\beta}{\beta}} \vol_{\bM}\right)^{\beta/(1-\beta)} \\
&\le& \left( d\int_{\bM} u^p \vol_{\bM} \right)^{\beta} \left( \int_{\bM} J_f^q \vol_{\bM}\right)^{\beta/(1-\beta)} <\infty.
\end{eqnarray*}
This completes the proof.
\end{proof}

\section{Invariant conformal structure}

One of our main tools is an invariant conformal structure for the given uniformly quasiregular 
endomorphism. In what follows we introduce some terminology and basic observations on measurable conformal structures. 

A \emph{measurable conformal structure} $G$ on $\bM$ is a bounded measurable mapping $G \colon \bM \to S(\bM)$, where $S(\bM)$ is the space of positive definite symmetric bundle self-map on $T\bM$ of determinant $1$.

We say that a continuous $W^{1,n}_\loc$-mapping $f \colon \bM\to\bM$ is a
\emph{$G$-transformation} 
if it satisfies the \emph{Beltrami equation}
\begin{equation}
\label{eq:Beltrami}
D^tf(x) G(f(x)) Df(x) = J^{2/n}_f(x) G(x)\quad \mathrm{a.e.}\;x\in \bM.
\end{equation}

The fundamental result of Iwaniec and Martin \cite{IwaniecT:Quas} states that every uniformly quasiregular endomorphism of $\bM$ is a $G$-transformation for some measurable conformal structure $G$ on $\bM$. More precisely, they proved the following.

\begin{theorem}[Iwaniec-Martin \cite{IwaniecT:Quas}]
\label{thm:IM}
Let $\Gamma$ be an abelian quasiregular semigroup of endomorphisms of
$\bM$. Then there exists a measurable conformal structure $G_\Gamma$
on $\bM$ so that all mappings in
$\Gamma$ are $G_\Gamma$-transformations.
\end{theorem}

While Iwaniec and Martin prove the existence of an invariant conformal structure only for abelian quasiregular semigroups of $\bS^n$, their proof applies almost verbatim to all closed Riemannian manifolds $\bM$. Indeed, in the proof of \cite[Theorem 5.1]{IwaniecT:Quas}, instead of the canonical identification $T_x\bS^n = \R^n$ for $x\in \R^n$, where $\R^n \subset \hat \R^n = \bS^n$, it suffices to use point-wise isometries $T_x \bM \to \R^n$ given by some orthonormal basis of $T_x\bM$ at $x\in \bM$. We leave the details to the interested reader. 

Let $f\colon \bM\to\bM$ be a uniformly quasiregular endomorphism.
In what follows, we denote by $G=G_f$ 
a choice of a conformal structure $G_{\langle f \rangle}$ for $f$,
where $\langle f\rangle$ is the abelian semigroup generated by $f$.
Note that the conformal structure $G$ defines a measurable Riemannian metric
$\langle\cdot,\cdot\rangle_G$ on $\bM$ by 
\[
\langle v,w \rangle_G := \langle G(x)v , w\rangle
\]
for almost every $x \in \bM$ and all $v,w \in T_x \bM$. We denote by $|v|_G$ the norm of $v$ with respect to this inner product; we use the same notation for the induced inner products and norms for covectors. Furthermore, we denote 
\[
\norm{\xi}_{p,G} = \left( \int_{\bM} |\xi(x)|_G^p \vol_\bM(x) \right)^{1/p},
\]
where $\xi$ is a measurable form on $\bM$. Note that, since $\det G \equiv 1$, the measurable volume form $\vol_G$ for this Riemannian metric coincides with $\vol_\bM$ almost everywhere. We note also that, for $m$-forms $\xi$ on $\bM$, $\norm{\xi}_{n/m,G}$ is uniformly comparable to $\norm{\xi}_{n/m}$, since $\id \colon (\bM,\langle\cdot,\cdot\rangle_G) \to (\bM, \langle \cdot,\cdot \rangle)$ is quasiconformal.

The $G$-transformations $f^k$ are conformal mappings in the metric $\langle \cdot, \cdot \rangle_G$. Indeed, 
\begin{align*}
\langle Df(x)v, Df(x)w \rangle_G =& \langle G(f(x))Df(x)v, Df(x)w \rangle \\
=& \langle D^t f(x) G(f(x))Df(x)v,w\rangle \\
=& J_f^{2/n}(x) \langle G(x)v,w \rangle = J_f^{2/n}(x)\langle v,w\rangle_G
\end{align*}
for all $v,w\in T_x \bM$ and almost every $x\in \bM$.

The measurable Riemannian metric $\langle\cdot,\cdot\rangle_G$ defines an $n$-Laplace operator $\Delta_n=\Delta_{n,G}$ and, in what follows, we consider the equation $\Delta_n u = \star \mu$, where $\mu$ is a finite Borel measure on $\bM$, 
and the pull-back of this equation under a $G$-transformation. The weak form of the equation $\Delta_n u = \star \mu$ reads as 
\begin{equation}
\label{eq:Delta_weak}
\int_{\bM} \langle |\rd u|_G^{n-2}\rd u, \rd \psi \rangle_G \vol_\bM = \int_{\bM} \psi \dmu
\end{equation}
for all $\psi \in C^\infty(\bM)$, or by the density of $C^\infty(\bM)$ in $W^{1,n}(\bM)$, for all $\psi\in W^{1,n}(\bM)$.
We emphasize that the $n$-Laplace operator $\Delta_n$ has measurable coefficients and that the equation $\Delta_n u = \star \mu$ is, in fact, an $\cA$-harmonic equation $\rd^* \cA(\rd u) = \star \mu$ for the bundle map $\cA(x,\xi) = \langle G(x)\xi,\xi\rangle^{(n-2)/2} G(x)\xi$,  where $\xi\in T^*_x\bM$ and $x\in \bM$.

A standard argument for $\cA$-harmonic morphism property for quasiregular mappings shows that the equation $\Delta_n u = \star \mu$ is invariant under a $G$-transformation $f$, that is, $\Delta_n(u\circ f) = \star f^*\mu$ if $u$ is a solution to \eqref{eq:Delta_weak}: for, since $f$ is a $G$-transformation and $u$ is a solution to \eqref{eq:Delta_weak}, 
\begin{align*}
 \int_{\bM} \langle |\rd (u\circ f)|_G^{n-2} \rd (u\circ f), \rd\psi \rangle_G \vol_\bM
=& \int_{\bM} \langle |\rd u|_G^{n-2} \rd u, \rd (f_*\psi)\rangle_G \vol_\bM\\
=& \int_{\bM} f_*\psi \dmu = \int_{\bM} \psi \rd(f^*\mu)
\end{align*}
for all $u \in W^{1,n}(\bM)$ and all $\psi\in W^{1,n}(\bM)$; see \cite[p.\ 271]{HeinonenJ:Nonptd}.
We record this $G$-invariance in terms of $\cA$-harmonic equations.

\begin{lemma}
\label{lemma:A_f}
Let $f\colon \bM \to \bM$ be a uniformly quasiregular endomorphism. Then there exists a bundle map $\cA_f \colon T^*\bM \to T^*\bM$ so that for all $u\in W^{1,n}(\bM)$, 
\begin{gather*}
 f^* \rd^* \cA_f(\rd u) = \rd^* \cA_f(\rd f^*u)
\end{gather*}
as measures.
\end{lemma}


\section{A construction of equilibrium measure}\label{sec:construct}

Let $f\colon \bM \to \bM$ be a uniformly $K$-quasiregular endomorphism of degree $d>1$. Let $G$ be an invariant conformal structure for $f$ and let $\Delta_n$ be the associated measurable $n$-Laplace operator. Let $\cA=\cA_f$ be the bundle map determined by $\Delta_n$.

\begin{lemma}\label{lem:Sobolev}
For every $u\in W^{1,n}(\bM)$ and every $\psi\in W^{1,n}(\bM)$,
\begin{gather}
\label{eq:degree_control}
\left|\int_{\bM} \psi\star \rd^* \cA(\rd(f^*u))\right| \le
\norm{\rd\psi}_{n,G} \norm{\rd u}_{n,G}^{n-1}d^{(n-1)/n}.
\end{gather}
\end{lemma}

\begin{proof}
Since $f$ is a $G$-transformation, 
\[
\int_{\bM} \psi\star \rd^* \cA(\rd(f^*u)) = \int_{\bM} \langle |\rd (u\circ f)|^{n-2}_G \rd (u\circ f),\rd\psi \rangle  \vol_\bM
\]
for $u\in W^{1,n}(\bM)$ and $\psi\in W^{1,n}(\bM)$. Thus, 
\begin{align*}
 \left|\int_{\bM} \psi\star \rd^* \cA(\rd(f^*u))\right| 
 &\le \int_{\bM} |\rd(u\circ f)|_G^{n-1} |\rd\psi|_G \\
 &\le \norm{\rd\psi}_{n,G} \left( \int_{\bM} |\rd(u\circ f)|_G^n\right)^{(n-1)/n} \\
 &\le \norm{\rd\psi}_{n,G} \left( \int_{\bM} (|\rd u|_G^n\circ f) J_f \right)^{(n-1)/n} \\
 &= \norm{\rd\psi}_{n,G} \norm{\rd u}^{n-1}_{n,G}d^{(n-1)/n}
\end{align*}
by H\"older's inequality and the change of variables.
\end{proof}

\begin{theorem}\label{thm:limits_Lp}
There exists a probability measure $\mu_f$ on $\bM$ such that
\begin{gather}
 \frac{(f^k)^*\omega}{d^k} \weakto \mu_f\label{eq:equilibrium}
\end{gather}
as $k \to \infty$ whenever $\omega\in L^p(\bigwedge^n \bM)$ is an $n$-form so that $\int_{\bM} \omega = 1$.
Here $p=p(n,K)>1$ is as in Proposition $\ref{prop:standard_qr}$.
\end{theorem}

\begin{proof}
By Proposition \ref{prop:standard_qr}, there exists a solution $u\in W^{1,n}(\bM)$ to the equation
\begin{gather}
 \rd^*\cA(\rd u) = \star\left(\frac{f^*\omega}{d}-\omega\right).\label{eq:potential}
\end{gather}
Let $\psi \in C^{\infty}(\bM)$. By Lemmas \ref{lem:Sobolev} and \ref{lemma:A_f}, 
\begin{align*}
\left|\int_{\bM}\psi
 \rd\left(\frac{(f^{k+1})^*\omega}{d^{k+1}}-\frac{(f^k)^*\omega}{d^k}\right)\right| 
=&\frac{1}{d^k} \left|\int_{\bM}\psi \rd(f^k)^*\left(\frac{f^*\omega}{d}-\omega\right)\right| \\
=&\frac{1}{d^k} \left|\int_{\bM}\psi\star \rd^* \cA(\rd((f^k)^*u))\right|\\
\le&\norm{\rd\psi}_{n,G} \norm{\rd u}^{n-1}_{n,G}(d^k)^{(n-1)/n-1}\\
\le&\norm{\rd\psi}_{n,G} \norm{\rd u}^{n-1}_{n,G}d^{-k/n}
\end{align*}
for every $k\in \N$. Since $C^\infty(\bM)$ is dense in $C^0(\bM)$, the weak limit 
\begin{gather*}
 \mu:=\lim_{k\to\infty}\frac{(f^k)^*\omega}{d^k}
\end{gather*}
exists and satisfies
\begin{gather}
\label{eq:mu_limit}
 \left|\int_{\bM}\psi
 \rd\left(\mu-\frac{(f^k)^*\omega}{d^k}\right)\right|
 \le \norm{\rd\psi}_{n,G} \norm{\rd u}^{n-1}_{n,G}\frac{d^{-k/n}}{1-d^{-1/n}}
\end{gather} 
for all $\psi \in C^\infty(\bM)$ and for all $k\in \N$.

We show now that the measure $\mu$ does not depend on the choice of 
$\omega\in L^p(\bigwedge^n \bM)$. 
Let $\omega$ and $\omega'$ be $n$-forms in $L^p(\bigwedge^n \bM)$ so that 
\[
\int_{\bM} \omega = \int_{\bM} \omega' = 1.
\]
Then there exists a solution $u\in W^{1,n}(\bM)$ to the equation
\[
\rd^*\cA(\rd u) = \star\left(\omega'-\omega\right)
\]
and, by the same computation as the above, for all $\psi \in C^\infty(\bM)$,
\begin{gather*}
 \left|\int_{\bM}\psi \rd\left(\frac{(f^k)^*\omega'}{d^k}-\frac{(f^k)^*\omega}{d^k}\right)\right| 
 \le \norm{\rd\psi}_{n,G} \norm{\rd u}^{n-1}_{n,G}d^{-k/n} \to 0
\end{gather*}
as $k\to\infty$. Since $C^{\infty}(\bM)$ is dense in $C^0(\bM)$, the proof is complete.
\end{proof}

\section{Proof of Theorem \ref{thm:properties}}

For the proof of Theorem \ref{thm:properties} we need a variant of Rickman's Montel theorem (\cite[Corollary IV.3.14]{RickmanS:Quam}) in our setting; the argument is a well-known combination of a version of Rickman's Picard theorem and Zalcman's lemma; see \cite{HinkkanenA:Locuqr}. 

We denote by $B^n(r)$ the open ball in $\mathbb{R}^n$ of radius $r>0$ and around the origin.

\begin{proposition}[Rickman's Montel theorem]\label{prop:RM}
There exists $q=q(n,K)\in\mathbb{N}$ so that for every distinct $a_1,\ldots, a_q\in\bM$,
any family $\{F_j;j\in\mathbb{N}\}$ of $K$-quasiregular mappings $F_j\colon B^n\to\bM$ omitting $a_1,\ldots,a_q$ is normal.
\end{proposition}

\begin{proof}
By a result of Holopainen and Rickman \cite{HolopainenI:Ricchf}, 
we may fix $q=q(n,K)$ so that every $K$-quasiregular mapping $\R^n \to \bM$ omitting $q$ points 
is constant. 

Suppose now that $E=\{a_1,\ldots, a_q\}$ is a set of $q$ distinct
points on $\bM$ and suppose that 
a family $\{F_j\}$ of $K$-quasiregular mappings $B^n\to \bM\setminus E$ is not normal. 
Then, changing the affine coordinate if necessary,
by Zalcman's lemma, there are a sequence $(F_i)\subset\{F_j\}$,
a sequence $\rho_i\searrow 0$ and a sequence of points $x_i\in B^n(1/2)$ 
tending to the origin such that $(F_i)$ does not converge locally uniformly but
that the sequence of mappings 
\begin{gather*}
 \hat{F}_i(x): = F_i(x_i + \rho_i x)
\end{gather*}
from $B^n((1/2)/\rho_i)$ to $\bM$
converges to a non-constant $K$-quasiregular mapping 
$\hat{F}\colon \R^n \to \bM$ locally uniformly; 
see \cite[Section 3]{HinkkanenA:Locuqr} or \cite[Section 2]{BonkM:Quamc}. 

Since all $\hat{F}_i$ omit the set $E$, we have a Hurwitz type result that $\hat{F}$ also omits the set $E$: for, if this is not the case, then there exist $x_0 \in \hat{F}^{-1}(E)$ and a normal neighborhood $\Omega$ of $x_0$. Since $\hat{F}_i|\overline{\Omega} \to \hat{F}|\overline{\Omega}$ uniformly as $i \to \infty$, 
\[
i(x_0,\hat{F}) = \deg(\hat{F}(x_0),\hat{F}, \Omega) = \lim_{i \to \infty} \deg(\hat{F}(x_0),\hat{F}_i,\Omega) = 0,
\]
which is a contradiction; see \cite[Section I.4]{RickmanS:Quam} for notation and terminology.

Thus $\hat{F}$ is constant by the choice of $q$. This is a contradiction. 
\end{proof}

Let $\mathcal{E}(f)$ the subset of all points in $\bM$ having a finite backward orbit under $f$, that is, 
\[
\mathcal{E}(f) = \{x \in \bM ; \# \bigcup_{k\ge 0} f^{-k}(x) <\infty\}.
\]
Then any $a\in\mathcal{E}(f)$ is periodic under $f$ and satisfies
$i(a,f)=d$. By an estimate from \cite[Theorem III.4.7]{RickmanS:Quam},
any $a\in\mathcal{E}(f)$ is a superattracting periodic point of $f$, and in particular,
\begin{gather*}
 \mathcal{E}(f)\cap\mathcal{J}(f)=\emptyset;
\end{gather*}
see \cite[p.83]{HinkkanenA:Locuqr}.
Since the following expansion property of $\mathcal{J}(f)$ holds,
by Rickman's Montel theorem, $\mathcal{E}(f)$ 
has the cardinality at most $q=q(n,K)$ in Proposition \ref{prop:RM}.

\begin{proposition}\label{prop:HMM}
Let $f\colon \bM \to \bM$ be a uniformly quasiregular endomorphism. Then 
for every $x_0\in\mathcal{J}(f)$ and
every neighborhood $\Omega$ of $x_0$ $($in $\bM)$ small enough,
\begin{gather*}
 \bigcup_{k\in\mathbb{N}}f^k(\Omega)=\bM\setminus\mathcal{E}(f).
\end{gather*}
\end{proposition}  

\begin{proof}
 For each $x_0\in\mathcal{J}(f)$ and
 each neighborhood $\Omega$ of $x_0$,
 the subset $\bM\setminus\bigcup_{k\in\mathbb{N}}f^k(\Omega)$
 is completely invariant under $f$, and is (possibly empty and) finite
 by Rickman's Montel theorem.
 Hence $\bM\setminus\bigcup_{k\in\mathbb{N}}f^k(\Omega)\subset\mathcal{E}(f)$,
 that is, $\bigcup_{k\in\mathbb{N}}f^k(\Omega)\supset\bM\setminus\mathcal{E}(f)$.

 If there is $x_0\in\mathcal{J}(f)$ such that
 any neighborhood $\Omega$ of $x_0$ satisfies
 $\bigcup_{k\in\mathbb{N}}f^k(\Omega)\not\subset\bM\setminus\mathcal{E}(f)$, then
 from $f^{-1}(\mathcal{E}(f))\subset\mathcal{E}(f)$, 
 we have $x_0\in\overline{\mathcal{E}(f)}=\mathcal{E}(f)$. This contradicts
 $\mathcal{E}(f)\cap\mathcal{J}(f)=\emptyset$. 
\end{proof}

Having Proposition \ref{prop:HMM} at our disposal, 
we close this section with a proof of Theorem \ref{thm:properties}. 

\begin{proof}[Proof of Theorem $\ref{thm:properties}$]
Let $\omega\in C^{\infty}(\bigwedge^n\bM)$. 
Let $B=B(x,3r)$ be an open ball in $\bM\setminus\mathcal{J}(f)$.
Then $\{f^k|B;k\in\mathbb{N}\}$ is normal, and by passing to a subsequence, 
we may assume that $(f^k|B)$ converges locally uniformly to a $K$-quasiregular mapping 
$F\colon B \to \bM$. Since
$\int_{B(x,2r)} F^*\omega < \infty$, by \cite[Lemma VI.8.8]{RickmanS:Quam}, 
\[
\mu_f(B(x,r)) = \lim_{k \to \infty} \frac{1}{d^k} \int_{B(x,r)} (f^k)^*\omega \le \lim_{k \to \infty} \frac{1}{d^k} \int_{B(x,2r)} F^*\omega = 0.
\]
Thus $\spt \mu_f \subset \mathcal{J}(f)$. 
In particular, by $\mathcal{E}(f)\cap\mathcal{J}(f)=\emptyset$,
we can choose some $a\in\spt\mu_f\setminus\mathcal{E}(f)$. 
By $f^{-1}(\spt\mu_f)\subset\spt\mu_f$, Proposition \ref{prop:HMM}
also concludes that $\mathcal{J}(f)\subset\overline{\{f^{-k}(a);k\in\mathbb{N}\}}\subset\spt\mu_f$.

Suppose that $\mu_f$ has an atom $a\in\bM$. 
Then $a\in\spt\mu_f=\mathcal{J}(f)$. Since $\mu_f$ is invariant under $f$, 
\begin{gather*}
 \mu_f(\{f(a)\})=(f_*\mu_f)(\{f(a)\})
=\mu_f(f^{-1}(f(a)))\ge\mu_f(\{a\}),
\end{gather*}
so $\#\{f^k(a);k\in\mathbb{N}\}<\infty$ by $\mu_f(\bM)=1$.
Without loss of generality, we may assume that $f(a)=a$. 
Then since $\mu_f$ is balanced under $f$,
\begin{gather*}
\mu_f(\{a\})=\frac{f^*\mu_f}{d}(\{a\})=\frac{i(a,f)\mu_f(\{a\})}{d},
\end{gather*}
which with $\mu_f(\{a\})>0$ implies $i(a,f)=d$. Hence $f^{-1}(a)=\{a\}$,
so $a\in\mathcal{E}(f)$. This contradicts 
$\mathcal{E}(f)\cap\mathcal{J}(f)=\emptyset$.
\end{proof}


\section{Approximation of the equilibrium measure}

Let $f\colon \bM \to \bM$ be a uniformly $K$-quasiregular endomorphism of degree $d>1$. 

\begin{theorem}\label{th:equidist}
The subset
\begin{gather*}
 E(f):=\left\{a\in\bM;\frac{(f^k)^*\delta_a}{d^k}\not\weakto\mu_f\text{ as }k\to\infty\right\}
\end{gather*}
is of Hausdorff dimension at most $n-1$. 
\end{theorem}

\begin{rem}
We note that $\mathcal{E}(f)\subset E(f)$ since for every $a\in\mathcal{E}(f)$,
 any weak limit of $(f^k)^*\delta_a/d^k$ has an atom, but $\mu_f$ has no atom
 by Theorem \ref{thm:properties}.

 Theorem \ref{th:equidist} gives a characterization of $\mu_f$
 as the unique balanced probability measure $\nu$ under $f$ whose support is 
 in $\bM\setminus E(f)$: for, also using the Fubini theorem,
\begin{multline*}
 \int_{\bM}\phi\rd\nu=\int_{\bM}\phi\rd\frac{(f^k)^*\nu}{d^k}
=\int_{\bM}\frac{(f^k)_*\phi}{d^k}\rd\nu
=\int_{\bM}(\int_{\bM}\phi\rd\frac{(f^k)^*\delta_a}{d^k})\rd\nu(a)\\
=\int_{\bM\setminus E(f)}(\int_{\bM}\phi\rd\frac{(f^k)^*\delta_a}{d^k})\rd\nu(a)
 \weakto\int_{\bM\setminus E(f)}(\int_{\bM}\phi\rd\mu_f)\rd\nu
=\int_{\bM}\phi\rd\mu_f
 \end{multline*}
 as $k\to\infty$ if $\phi\in C^0(\bM)$. 
\end{rem}

For the proof of Theorem \ref{th:equidist}
we recall a well-known Riesz-potential estimate based on Green's functions; we refer to \cite[Chapter IV \S 2.3]{AubinT:Somnpr} for more details. 

\begin{lemma}
\label{lemma:RP}
For every $a\in\bM$ and every continuous $v\in W^{1,n}(\bM)$,
\begin{gather*}
\left| v(a)- \vint_{\bM} v(y) \vol_{\bM} \right| 
\le \int_{\bM} \frac{|\grad v|(y)}{|y-a|^{n-1}}\vol_{\bM}.
\end{gather*}
\end{lemma}

\begin{proof}
Suppose first that $v\in W^{1,n}(\bM)$ and set
\[
u:=v-\vint_{\bM} v(y)\vol_{\bM}.
\]
For every $a\in\bM$, by \cite[Theorem 4.13]{AubinT:Somnpr} and integration by parts, we have
\[
|u(a)| \le \int_{\bM} \langle \grad_y G(a,y), \grad u(y) \rangle \vol_{\bM}
\]
for every $u\in C^2(\bM)$ of zero mean on $\bM$. Here $G(\cdot, \cdot)$ is Green's function on $\bM$. Moreover, there exists $C>0$ so that $|\grad_y G(y,a)| \le C |y-a|^{1-n}$ 
for all distinct $a,y\in \bM$. Thus
\[
|u(a)| \le \int_{\bM} \frac{|\grad u(y)|}{|y-a|^{n-1}} \vol_{\bM}
\]
for every $u\in C^2(\bM)$ of zero mean on $\bM$.

In the general case of $v\in W^{1,n}(\bM)$, 
we may approximate, by a standard convolution argument, $v$ by $C^\infty$-smooth functions $v_k$ so that $v_k\to v$ uniformly and $\norm{\grad v-\grad v_k}_{L^1} \to 0$ as $k\to \infty$. The claim now follows.
\end{proof}

A key part of the proof of Theorem \ref{th:equidist} is the following capacity estimate.

\begin{lemma}\label{lemma:cap_key}
Let $\omega=\vol_{\bM}/\vol_{\bM}(\bM)$. 
Then for every $\phi\in C^\infty(\bM)$, every $\varepsilon>0$ and every $k\ge 0$, the subset
\[ 
K_{\varepsilon,k}(\phi):=\left\{a\in \bM \colon \left|\int_{\bM}\phi\rd\frac{(f^k)^*(\delta_a-\omega)}{d^k}\right| \ge \varepsilon \right\}
\]
is compact and 
\[
C_1\left(K_{\varepsilon,k}(\phi)\right)\le \frac{K^{1/n}\norm{\grad \phi}_n}{\varepsilon d^{k/n}}.
\]
\end{lemma}

\begin{proof}
For every $a\in\bM$, every $\varepsilon>0$ and $k\ge 0$, since
\[
\int_{\bM}\phi\rd\frac{(f^k)^*(\delta_a-\omega)}{d^k}
= \frac{1}{d^k}((f^k)_*\phi)(a)
- \int_{\bM} (f^k)_*\phi(y) \rd\omega,
\]
the set $K_{\varepsilon,k}(\phi)$ is clearly closed, so compact.

To show the second claim, let $\varepsilon>0$, $k\in \N$, and $\phi\in C^\infty(\bM)$. We may assume that $C_1(K_{\varepsilon,k}(\phi))>0$. Let $\mu$ be an equilibrium measure of $K_{\varepsilon,k}(\phi)$ with respect to $1$-Riesz kernel on $\bM$. Let $\Omega = \bM\setminus f^k(B_{f^k})$, where $B_{f^k}$ is the \emph{branch set of $f^k$}, that is, the set where $f^k$ fails to be a local homeomorphism. 

Since $\bM$ is closed, we may cover $\bM$ with finitely many $2$-bilipschitz charts. Then, applying the Vitali covering theorem locally in charts, we may cover $\Omega$, up to a 
$\omega$-nullset, with a countable collection $\{B_i\}$ 
of domains that are $2$-bilipschitz equivalent to Euclidean $n$-balls and are chosen so that $f^k$ is univalent in components of $(f^k)^{-1}(B_i)$. 
 
We denote by $g^i_j$ the inverses $B_i \to U^i_j$, 
where $U^i_j$ is a component of $(f^k)^{-1}(B_i)$ ($j\in\{1,\ldots, d^k\}$).
For every $j\in\{1,\ldots, d^k\}$, we define the mapping
\begin{gather*}
 g_j :\bigcup_i B_i \to \bM 
\end{gather*}
such that $g_j|B_i = g^i_j$ for each $i$. 
Put $U_j:=g_j(\bigcup_i B_i)$. 
Then $|\bM \setminus \bigcup_j U_j| =0$ and $U_j \cap U_{j'}=\emptyset$ if $j\ne j'$.

Put $v = v_{\phi}:=(f^k)_* \phi/d^k$. Then 
\begin{eqnarray*}
\varepsilon &\le&\int_{K_{\varepsilon,k}(\phi)}\left|\int_\bM \phi
\rd\frac{(f^k)^*(\delta_a-\omega)}{d^k}\right| \dmu(a) \\
&\le& \int_{\bM} \left| v(a)-\vint_{\bM} v(y) \vol_{\bM} \right| \dmu(a).
\end{eqnarray*}
Since $f^k(B_{f^k})$ has zero measure, by Lemma \ref{lemma:RP} 
and the change of variables, for every $j\in\{1,\ldots, d^k\}$, 
\begin{align*}
\varepsilon \le& \int_{\bM} \left( \int_{\bM} \frac{|\grad v|(y)}{|a-y|^{n-1}} \vol_{\bM} \right) \rd\mu(a) \\
=& \int_{\bM} \left( \int_{\Omega} \frac{|\grad v|(y)}{|a-y|^{n-1}} \vol_{\bM} \right) \rd\mu(a) \\
=& \int_{\bM} \left( \int_{U_j} \frac{|\grad v|(f^k(x))}{|a-f^k(x)|^{n-1}}J_{f^k}(x)\vol_{\bM}\right) \rd\mu(a) \\
=& \int_{U_j} \left( \int_{\bM} \frac{\dmu(a)}{|a-f^k(x)|^{n-1}} \right) |\grad v|(f^k(x))J_{f^k}(x) \vol_{\bM} \\
=& \int_{U_j}  U^\mu(f^k(x)) |\grad v|(f^k(x))J_{f^k}(x)\vol_{\bM} \\
\le& C_1(K_{\varepsilon,k}(\phi))^{-1} \int_{U_j}  |\grad v(x)|(f^k(x))J_{f^k}(x) \vol_{\bM}.
\end{align*}
Here the last inequality follows from \eqref{eq:upper}. Hence
\begin{eqnarray*}
C_1(K_{\varepsilon,k}(\phi)) &\le& \frac{1}{\varepsilon d^k} \sum_{j=1}^{d^k} \int_{U_j} |\grad v(x)|(f^k(x))J_{f^k}(x) \vol_{\bM} \\
&=& \frac{1}{\varepsilon d^k} \int_{\bM} |\grad v(x)|(f^k(x))J_{f^k}(x)\vol_{\bM}.
\end{eqnarray*}
On the other hand,
\begin{align*}
|\grad v|(f^k(x)) J_{f^k}(x) \le& \frac{1}{d^k} \sum_{j=1}^{d^k} \left| \grad(\phi \circ g_j)(f^k(x))\right| J_{f^k}(x) \\
\le&\frac{1}{d^k} \sum_{j=1}^{d^k} |Dg_j|(f^k(x)) |\grad \phi(x)|(g_j(f^k(x))) J_{f^k}(x)\\
\le& K^{1/n}\frac{1}{d^k} \sum_{j=1}^{d^k} |\grad \phi(x)| J_{g_j}^{1/n}(f^k(x))  J_{f^k}(x) \\
\le& K^{1/n}\frac{1}{d^k} \sum_{j=1}^{d^k} |\grad \phi(x)| J_{f^k}^{(n-1)/n}(x)\\
=& K^{1/n} |\grad \phi(x)| J_{f^k}^{(n-1)/n}(x)
\end{align*}
for almost every $x\in \bM$. Thus
\begin{align*}
\frac{1}{d^k}\int_{\bM} |\grad v|(f^k(x))J_{f^k}(x) \vol_{\bM} \le& \frac{K^{1/n}}{d^k}\int_{\bM}  |\grad \phi(x)| J_{f^k}^{(n-1)/n}(x) \vol_{\bM} \\
\le& \frac{K^{1/n}\norm{\grad \phi}_n}{d^k} \left( \int_{\bM}J_{f^k}(x) \vol_{\bM} \right)^{(n-1)/n} \\
\le& \frac{K^{1/n}\norm{\grad \phi}_n}{d^k}d^{k(n-1)/n} \\
\le& \frac{K^{1/n}\norm{\grad \phi}_n}{d^{k/n}},
\end{align*}
which completes the proof.
\end{proof}

\begin{proof}[Proof of Theorem $\ref{th:equidist}$]
Let $\omega=\vol_{\bM}/\vol_{\bM}(\bM)$. By Theorem \ref{thm:limits_Lp},
for every $a\in\bM$, we have
\begin{gather*}
 \limsup_{k\to\infty}\left|\int_{\bM}\phi\rd\left(\frac{(f^k)^*\delta_a}{d^k}-\mu_f\right)\right|
=\limsup_{k\to\infty}\left|\int_{\bM}\phi
\rd\frac{(f^k)^*(\delta_a-\omega)}{d^k}\right|.
\end{gather*}
For every $\phi\in C^{\infty}(\bM)$, we show that for every $\epsilon>0$, the set
\begin{gather*}
 E(\phi)
:=\left\{a\in\bM:\limsup_{k\to\infty}\left|\int_{\bM}\phi
\rd\frac{(f^k)^*(\delta_a-\omega)}{d^k}\right|>0\right\}
\end{gather*}
has the $(n-1+\epsilon)$-dimensional Hausdorff measure $0$. Then the proof
will be complete since a countable subset of $C^{\infty}(\bM)$ is dense in $C(\bM)$ and (\ref{eq:equilibrium}) holds.

Suppose that for some $\epsilon>0$, $\haus^{(n-1+\epsilon)}(E(\phi))>0$. 
Then there exists a compact subset $K'\subset E(\phi)$ of 
positive $(n-1+\epsilon)$-dimensional Hausdorff measure. 
For every $i\in\mathbb{N}$ and every $k\in\mathbb{N}$, we denote 
\begin{gather*}
 K'_{1/i,k}(\phi):=\left\{a\in K'\colon 
\left|\int_{\bM}\phi\rd\frac{(f^k)^*(\delta_a-\omega)}{d^k}\right|
\ge\frac{1}{i}\right\} = K'\cap K_{1/i,k}(\phi).
\end{gather*}
Then $K'_{1/i,k}(\phi)$ is compact and, by the monotonicity of capacity 
and Lemma \ref{lemma:cap_key}, 
\begin{gather*}
C_1(K'_{1/i,k}(\phi)) \le C_1(K_{1/i,k}(\phi)) \le i \frac{K^{1/n}\norm{\grad \phi}_n}{d^{k/n}}.
\end{gather*}
Thus the Hausdorff $(n-1+\epsilon)$-content $\lambda_{n-1+\epsilon}(K'_{1/i,k})$ of $K'_{1/i,k}$
is estimated as 
\begin{gather*}
 \lambda_{n-1+\epsilon}(K'_{1/i,k}(\phi))\le\left(i\frac{K^{1/n}\norm{\grad \phi}_\infty}{\delta d^{k/n}}\right)^{\frac{n-1+\epsilon}{n-1}}
\end{gather*}
for every $i\in\mathbb{N}$; see \cite[Corollary III.4.2]{LandkofN:Foumpt}. Thus 
\begin{align*}
\lambda_{n-1+\epsilon}(E(\phi))
=&\lambda_{n-1+\epsilon}\left(\bigcup_{i\in\mathbb{N}}\bigcap_{N\in\mathbb{N}}\bigcup_{k\ge N}K'_{1/i,k}(\phi)\right)\\
\le&\sum_{i\in\mathbb{N}}\lim_{N\to\infty}\sum_{k\ge N}\lambda_{n-1+\epsilon}(K'_{1/i,k}(\phi))=0.
\end{align*}
On the other hand, the assumption $\haus^{(n-1+\epsilon)}(E(\phi))>0$ implies 
that $\lambda_{n-1+\epsilon}(E(\phi))>0$. This is a contradiction. 

Now the proof is complete.
\end{proof}

\section{Proof of Theorem \ref{thm:mixing}}

We dedicate this section for the proof of the last result mentioned in Section \ref{sec:introduction}, Theorem \ref{thm:mixing}. In what follows, $G$ is an invariant conformal structure for the uniformly quasiregular endomorphism $f$ of $\bM$. In particular, $f$ is a $G$-transformation.

Suppose first that $\psi$ and $\phi$ are $C^\infty$-functions on $\bM$, and let $\omega\in C^\infty(\bigwedge^n \bM)$ be an $n$-form satisfying 
\[
\int_{\bM} \omega = 1.
\]
As in the proof of Theorem \ref{thm:limits_Lp}, let $u$ be a solution to the equation 
\[
\rd^*\cA_f(\rd u) = \star \left( \frac{f^*\omega}{d} - \omega\right)
\]
as in \eqref{eq:potential}. The first step is to show that 
\begin{equation}
\label{eq:psi_phi}
\left| \int_{\bM} \psi (\phi\circ f^k) \rd\left(\mu_f  - \frac{(f^{k+m})^* \omega}{d^{k+m}}\right) \right| 
\le C \norm{\rd u}^{n-1}_{n,G}d^{-m/n}
\end{equation}
for all $k,m\in \N$, where $C$ depends only on $\psi$, $\phi$ and $n$.

Since $f$ is a $G$-transformation, for all $k\in\mathbb{N}$,
\begin{eqnarray*}
\norm{\rd(\phi \circ f^k)}^n_{n,G} 
&=& \int_{\bM} |\rd(\phi\circ f^k)|^n_G \vol_{\bM} 
= \int_{\bM} (|\rd\phi|_G^n \circ f^k) J_{f^k} \vol_{\bM} \\
&=& d^k \int_{\bM} |\rd\phi|_G^n \vol_{\bM} 
= d^k \norm{\rd\phi}^n_{n,G}, 
\end{eqnarray*}
where the third equality follows from the change of variables. Thus
\begin{eqnarray*}
\norm{\rd(\psi (\phi\circ f^k))}_{n,G} &\le&  \norm{ (\phi\circ f^k) \rd\psi}_{n,G} + \norm{\psi \rd(\phi \circ f^k)}_{n,G} \\
&\le& \norm{\phi}_\infty \norm{\rd\psi}_{n,G} + \norm{\psi}_\infty \norm{\rd\phi}_{n,G} d^{k/n} \\
&\le& C' d^{k/n}
\end{eqnarray*}
for all $k\in\mathbb{N}$, where $C'>0$ depends only on $\psi$ and $\phi$. 
Thus, by \eqref{eq:mu_limit}, we have
\begin{eqnarray*}
\left| \int_{\bM} \psi (\phi\circ f^k) \rd\left(\mu_f  - \frac{(f^{k+m})^* \omega}{d^{k+m}}\right) \right| 
&\le& \norm{\rd (\psi (\phi\circ f^k))}_{n,G} \norm{\rd u}^{n-1}_{n,G} \frac{d^{-(k+m)/n}}{1-d^{-(k+m)/n}} \\ 
&\le& C' \norm{\rd u}^{n-1}_{n,G}\frac{d^{-m/n}}{1-2^{-2/n}},
\end{eqnarray*} 
which proves \eqref{eq:psi_phi} by choosing $C=C'/(1-2^{-2/n})$.

To show \eqref{eq:mixing}, let $\varepsilon>0$. Having \eqref{eq:psi_phi} at our disposal, we may fix $m_0>0$ so that 
\[
\left| \int_{\bM} \psi (\phi\circ f^k) \rd\left(\mu_f  - \frac{(f^{k+m})^* \omega}{d^{k+m}}\right) \right|< \varepsilon/2
\]
for all $k \in \N$ and all $m\ge m_0$. 

For each $m\ge m_0$, let us denote 
\begin{gather*}
 \tau =\tau_m= \frac{(f^m)^*\omega }{d^m}
\end{gather*}
for brevity. By the higher integrability of quasiregular mappings, we have, as in the proof of Proposition \ref{prop:standard_qr}, that $\tau \in L^p(\bigwedge^n \bM)$ for some $p>1$. Since $\phi \tau \in L^p(\bigwedge^n \bM)$, we conclude, by Theorem \ref{thm:limits_Lp}, that 
\[
\frac{(f^k)^*(\phi \tau)}{d^k} \weakto \left( \int_{\bM} \phi\rd\tau\right) \mu_f
\]
as $k\to \infty$. Thus
\begin{eqnarray*}
\int_{\bM} \psi (\phi\circ f^k)\rd \frac{(f^{k+m})^*\omega}{d^{k+m}} 
&=& \int_{\bM} \psi(\phi \circ f^k)\rd \frac{(f^k)^*\tau}{d^k} \\
&=& \int_{\bM} \psi\rd \frac{(f^k)^*(\phi \tau)}{d^k} 
\to \left( \int_{\bM} \psi \dmu_f \right) \left( \int_{\bM} \phi \rd\tau \right)
\end{eqnarray*}
as $k\to \infty$. 

In particular, there exists $k_0=k_0(m)\ge 0$ so that 
\[
\left| \int_{\bM} \psi (\phi\circ f^k)  \dmu_f  - \left( \int_{\bM} \psi \dmu_f \right) \left( \int_{\bM} \phi\rd \frac{(f^m)^*\omega}{d^m}\right)\right| < \varepsilon
\]
if $k\ge k_0$ and $m\ge m_0$. Since 
\[
\frac{(f^m)^*\omega}{d^m}\weakto \mu_f
\]
as $m\to \infty$, we obtain \eqref{eq:mixing} for smooth functions $\psi$ and $\phi$ on $\bM$. 

The equality \eqref{eq:mixing} for $\psi,\phi\in C^\infty(\bM)$ together with standard approximation arguments yield
\[
\mu_f(f^{-k}(A)\cap B) \to \mu_f(A)\mu_f(B)
\]
for all Borel sets $A$ and $B$ in $\bM$ and \emph{a fortiori} 
\eqref{eq:mixing} for all $\psi,\phi \in L^2(\mu_f)$. In particular, $\mu_f$ is ergodic; 
see e.g.\;\cite[Section 2.5]{PetersenK:Ergt}. The proof of Theorem \ref{thm:mixing} is complete. \qed


\begin{thebibliography}{10}

\bibitem{AstolaL:Latmcm}
L.~Astola, R.~Kangaslampi, and K.~Peltonen.
\newblock Latt{\`e}s-type mappings on compact manifolds.
\newblock {\em Conform. Geom. Dyn.}, 14:337--367, 2010.

\bibitem{AubinT:Somnpr}
T.~Aubin.
\newblock {\em Some nonlinear problems in {R}iemannian geometry}.
\newblock Springer Monographs in Mathematics. Springer-Verlag, Berlin, 1998.

\bibitem{BonkM:Quamc}
M.~Bonk and J.~Heinonen.
\newblock Quasiregular mappings and cohomology.
\newblock {\em Acta Math.}, 186(2):219--238, 2001.

\bibitem{DOnofrioL:Notpha}
L.~D'Onofrio and T.~Iwaniec.
\newblock Notes on {$p$}-harmonic analysis.
\newblock In {\em The {$p$}-harmonic equation and recent advances in analysis},
  volume 370 of {\em Contemp. Math.}, pages 25--49. Amer. Math. Soc.,
  Providence, RI, 2005.

\bibitem{DrasinD:Equnt}
D.~Drasin and Y.~Okuyama.
\newblock Equidistribution and {N}evanlinna theory.
\newblock {\em Bull. Lond. Math. Soc.}, 39(4):603--613, 2007.

\bibitem{HajlaszP:Weadmb}
P.~Haj{\l}asz, T.~Iwaniec, J.~Mal{\'y}, and J.~Onninen.
\newblock Weakly differentiable mappings between manifolds.
\newblock {\em Mem. Amer. Math. Soc.}, 192(899):viii+72, 2008.

\bibitem{HeinonenJ:Nonptd}
J.~Heinonen, T.~Kilpel{\"a}inen, and O.~Martio.
\newblock {\em Nonlinear potential theory of degenerate elliptic equations}.
\newblock Oxford Mathematical Monographs. The Clarendon Press Oxford University
  Press, New York, 1993.
\newblock Oxford Science Publications.

\bibitem{HinkkanenA:Locuqr}
A.~Hinkkanen, G.~J. Martin, and V.~Mayer.
\newblock Local dynamics of uniformly quasiregular mappings.
\newblock {\em Math. Scand.}, 95(1):80--100, 2004.

\bibitem{HolopainenI:Ricchf}
I.~Holopainen and S.~Rickman.
\newblock Ricci curvature, {H}arnack functions, and {P}icard type theorems for
  quasiregular mappings.
\newblock In {\em Analysis and topology}, pages 315--326. World Sci.
  Publishing, River Edge, NJ, 1998.

\bibitem{IwaniecT:LpiPDE}
T.~Iwaniec.
\newblock On {$L^{p}$}-integrability in {PDE}s and quasiregular mappings for
  large exponents.
\newblock {\em Ann. Acad. Sci. Fenn. Ser. A I Math.}, 7(2):301--322, 1982.

\bibitem{IwaniecT:Quas}
T.~Iwaniec and G.~Martin.
\newblock Quasiregular semigroups.
\newblock {\em Ann. Acad. Sci. Fenn. Math.}, 21(2):241--254, 1996.

\bibitem{IwaniecT:Nonhtm}
T.~Iwaniec, C.~Scott, and B.~Stroffolini.
\newblock Nonlinear {H}odge theory on manifolds with boundary.
\newblock {\em Ann. Mat. Pura Appl. (4)}, 177:37--115, 1999.

\bibitem{LandkofN:Foumpt}
N.~S. Landkof.
\newblock {\em Foundations of modern potential theory}.
\newblock Springer-Verlag, New York, 1972.
\newblock Translated from the Russian by A. P. Doohovskoy, Die Grundlehren der
  mathematischen Wissenschaften, Band 180.

\bibitem{LjubichM:Entpre}
M.~J. Ljubich.
\newblock Entropy properties of rational endomorphisms of the {R}iemann sphere.
\newblock {\em Ergodic Theory Dynam. Systems}, 3(3):351--385, 1983.

\bibitem{MartinG:GenLp}
G.~Martin, V.~Mayer, and K.~Peltonen.
\newblock The generalized {L}ichnerowicz problem: uniformly quasiregular
  mappings and space forms.
\newblock {\em Proc. Amer. Math. Soc.}, 134(7):2091--2097 (electronic), 2006.

\bibitem{MartioO:Onidq}
O.~Martio.
\newblock On the integrability of the derivative of a quasiregular mapping.
\newblock {\em Math. Scand.}, 35:43--48, 1974.

\bibitem{MartioO:Lusin}
O.~Martio and W.~P. Ziemer.
\newblock Lusin's condition ({N}) and mappings with nonnegative {J}acobians.
\newblock {\em Michigan Math. J.}, 39(3):495--508, 1992.

\bibitem{MattilaP:Avecfq}
P.~Mattila and S.~Rickman.
\newblock Averages of the counting function of a quasiregular mapping.
\newblock {\em Acta Math.}, 143(3-4):273--305, 1979.

\bibitem{MeyersN:Resrfs}
N.~G. Meyers and A.~Elcrat.
\newblock Some results on regularity for solutions of non-linear elliptic
  systems and quasi-regular functions.
\newblock {\em Duke Math. J.}, 42:121--136, 1975.

\bibitem{PeltonenK:Exauqm}
K.~Peltonen.
\newblock Examples of uniformly quasiregular mappings.
\newblock {\em Conform. Geom. Dyn.}, 3:158--163 (electronic), 1999.

\bibitem{PetersenK:Ergt}
K.~Petersen.
\newblock {\em Ergodic theory}, volume~2 of {\em Cambridge Studies in Advanced
  Mathematics}.
\newblock Cambridge University Press, Cambridge, 1989.
\newblock Corrected reprint of the 1983 original.

\bibitem{RickmanS:Quam}
S.~Rickman.
\newblock {\em Quasiregular mappings}, volume~26 of {\em Ergebnisse der
  Mathematik und ihrer Grenzgebiete (3) [Results in Mathematics and Related
  Areas (3)]}.
\newblock Springer-Verlag, Berlin, 1993.

\end{thebibliography}
\end{document}